\documentclass[12pt,a4paper]{amsart}

\usepackage[T1]{fontenc}
\usepackage[utf8]{inputenc}
\usepackage[english]{babel}
\usepackage{amsfonts, amssymb, amsmath, amsthm}
\usepackage[all]{xy}
\usepackage{enumerate}
\usepackage{pgf}
\usepackage{tikz}

\usepackage{color}

\usepackage{graphicx}
\usepackage{psfrag}
\usepackage{geometry}
\usepackage[colorlinks=true,pdfstartview=FitV,linkcolor=blue,citecolor=green,urlcolor=black,filecolor=magenta]{hyperref}
\usepackage{verbatim}

\newtheorem{theorem}{Theorem}[section]
\newtheorem{definition}[theorem]{Definition}
\newtheorem{prop}[theorem]{Proposition}
\newtheorem{lemma}[theorem]{Lemma}
\newtheorem{cor}[theorem]{Corollary}

\newcommand{\Aa}{{\mathcal A}}

\newcommand{\Hh}{{\mathcal H}}

\newcommand{\Ll}{{\mathcal L}}

\newcommand{\Rr}{{\mathcal R}}

\newcommand{\id}{{\mathrm{id}}}

\renewcommand{\sf}{{\mathfrak s}}

\newcommand{\R}{\mathbb{R}}
\newcommand{\N}{\mathbb{N}}
\newcommand{\Z}{\mathbb{Z}}

\newcommand{\M}{\mathcal M}

\newcommand{\im}{\mbox{\rm im}}

\title{Complete regularity of Ellis semigroups of $\mathbb Z$-actions}

\author{Marcy Barge}
\address{Montana State University,
Department of Mathematical Sciences,
Bozeman, MT 59717, USA}
\email{barge@math.montana.edu}

\author{Johannes Kellendonk}
\address{Univerisit\'{e} de Lyon, Universit\'{e} Claude Bernard Lyon 1, Institute Camille Jordan, CNRS UMR 5208, 69622 Villeurbanne, France}
\email{kellendonk@math.univ-lyon1.fr}


\begin{document}

\maketitle

\begin{abstract}
It is shown that the Ellis semigroup of a $\mathbb Z$-action on a compact totally disconnected space is  completely regular if and only if forward proximality coincides with forward asymptoticity and backward proximality coincides with backward asymptoticity. Furthermore, the Ellis semigroup of a $\mathbb Z$- or $\mathbb R$-action for which forward proximality and backward proximality are transitive relations is shown to have  at most two left minimal ideals. Finally, the notion of near simplicity of the  Ellis semigroup is introduced and related to the above.
\end{abstract}
\section{Introduction}
Associated to any topological dynamical system $(X,\alpha,T)$ is a semigroup $E(X,T)$, its enveloping, or Ellis semigroup. This is a completion of the set 
$\{\alpha^t,\ t\in T\}$ of continuous surjections given by the action of the group (or semigroup) $T$.   
Its algebraic and topological properties reflect those of the dynamical system. 
We focus here on one algebraic property called {\em complete regularity}
with the aim to understand what complete regularity of the Ellis semigroup implies dynamically. 

By construction, $E(X,T)$ contains a unit, the identity map.
It contains furthermore a (unique) minimal {two-sided} ideal $\M(X,T)$.
$\M(X,T)$ is a completely simple semigroup and therefore a disjoint union of isomorphic groups.  We say that $E(X,T)$ is {\em nearly simple}, if all its non-invertible elements belong to its minimal two-sided ideal $\M(X,T)$.  
Examples of dynamical systems whose Ellis semigroups are nearly simple include 
Sturmian subshifts and subshifts defined by bijective substitutions \cite{ESBS}.
A semigroup is {\em completely regular} if it is a disjoint union of groups, but these groups do not have to be isomorphic. A nearly simple Ellis semigroup is completely regular. 
The Ellis semigroup of a dynamical system associated with a higher dimensional almost canonical cut and project tiling is completely regular \cite{Aujogue}, and it is nearly simple
if and only if the tiling has maximal complexity exponent \cite{Aujogue-Barge-Kellendonk-Lenz}.

If $T=\Z$ or $\R$ then we can focus on the forward dynamics, i.e.\ the dynamics under the semigroup $T^+$ of positive $t\in T$, and define 
the {\em adherence semigroup} $\Aa(X,T^+)$ as the subsemigroup of elements of $E(X,T)$ which are limits of nets $(\alpha^{t_\nu})$ with $\lim t_{\nu} = +\infty$.  $\Aa(X,T^+)$ contains the minimal two-sided ideal $\M(X,T^+)$ of
$E(X,T^+)$, where $E(X,T^+)$ is the completion of the homeomorphisms $\alpha^t$ with $t\geq 0$.

The semigroup $E(X,T)$ captures the proximality relation: Two points $x,y\in X$ are proximal if $\inf_{t\in T} d(\alpha^t(x),\alpha^t(y)) = 0$, and this is the case if and only if there exists $f\in E(X,T)$ such that $f(x) = f(y)$. If $T=\Z$ or $\R$ then by restricting the above infimum to $t\in T^+$ we obtain the {\em foward proximality} relation. 
Points $x,y\in X$ are called {\em forward asymptotic} if $\lim_{t\to+\infty} d(\alpha^t(x),\alpha^t(y)) = 0$.
In this context, the following algebraic characterisations are known. For the first two see, for example, 
\cite{Auslander,hindman}. The third characterisation is proven in  \cite{Blanchard} for $T=\Z$ but the proof given there carries over to $T=\R$.
\begin{enumerate}
\item $E(X,T)$ is a group if and only if the proximal relation is trivial in the sense that $x$ is proximal to $y$ only if $x=y$. Dynamical systems with this property are called {\em distal}.
Likewise,  $E(X,T^+)$ is a group if and only if the forward proximal relation is trivial.
\item
$E(X,T)$ contains a unique minimal left ideal if and only if the proximal relation is transitive.
Likewise,  $E(X,T^+)$ contains a unique minimal left ideal if and only if the forward proximality relation is transitive.
\item
$\Aa(X,T^+)$ is left simple (that is, has no proper left ideals) if and only if forward proximality agrees with forward asymptoticity. Dynamical systems with this property are  called {\em forward almost distal}. We will see that
in this case also $\M(X,T^+)$ is left simple and $E(X,T^+)$ is nearly simple.
\end{enumerate}
Intuitively speaking, two points are forward proximal if they become arbitrarily close under the forward dynamics and they are even asymptotic if they stay closer and closer under the forward dynamics. Asymptoticity implies proximality but not the other way around. A forward proximal pair which is not forward asymptotic is called a {\em forward Li-Yorke pair} \cite{Blanchard}. 

We show in this article that, when the space $X$ is {\em totally disconnected}, a forward Li-Yorke pair gives rise to an element in $E(X,\Z^+)$ which is not injective on its image.  
 By applying this result to the forward and the backward dynamics of a $\Z$-action we are able to characterise the $\Z$-actions on totally disconnected spaces which have a completely regular Ellis semigroup as those 
for which there are no forward and no backward Li-Yorke pairs (Thm.~\ref{thm-complete-regularity}). 
Furthermore, the notions of complete regularity and near simplicity coincide for 
 $\Z$-actions on totally disconnected spaces.
One implication of the above is true for any $\Z$-action and can be easily extended to 
$\R$ actions: The absence of forward and backward Li-Yorke pairs implies near simplicity. 

We show moreover that the Ellis semigroup $E(X,T)$ has at most two minimal left ideals if forward proximality and backward proximality are transitive (Thm.~\ref{thm-E1}).

In the final section we provide an explicit example of an Ellis semigroup which is not completely regular. This example is given by a substitution of constant length. It is 
backward almost distal, but not forward almost distal.
\section{Background on semigroups}
We provide some background material on semigroups. A general reference is \cite{howie1995fundamentals}.
A semigroup is a (non-empty) set with (associative) binary operation. We denote it multiplicatively $ab$.
A semigroup with a unit element is called a monoid. No semigroup in this article will have a zero element.

\subsection{Ideals and idempotents}
A (left, right, or two-sided) ideal of a semigroup $S$ is a non-empty subsemigroup $I\subset S$ satisfying $SI\subset I$, $IS\subset I$, $SI\cup IS\subset I$. 
(Left, right, or two-sided) ideals are ordered by inclusion.
Whereas the intersection of two left ideals may be empty this is not the case for the intersection of two two-sided ideals and therefore a minimal two-sided ideal of $S$ is unique, if it exists. This ideal is called the {\em kernel} of $S$ and must contain all minimal left (and all minimal right) ideals.

A semigroup is called (left, right, or two-sided) simple if it has no proper (left, right, or two-sided) ideal. Instead of two-sided simple we also say simple. Note that a left simple semigroup is simple, as a two-sided ideal is a left ideal.

The kernel $\M$ of a semigroup is simple, for if it contains an ideal $I$ and $a\in I$ then $\M a\M$ is an ideal of $S$ which is contained in $I$. By minimality of $\M$ we thus have $I=\M$. 

We recall three of the famous Green's relations. Given a semigroup $S$
we let  $S^1$ be the monoid which is $S$, if $S$ has a unit, or $S$ with added unit $1$, if it has none. 
Two elements $a,b\in S$ are in the same $\Ll$-class if they generate the same left ideal, $S^1a = S^1b$.  They are in the same 
 $\Rr$-class if they generate the same right ideal, $a S^1 = b S^1$.
Finally,  the $\Hh$-classes of $S$ are the intersections of the $\Ll$-classes with the $\Rr$-classes.

An idempotent of a semigroup $S$ is an element $p\in S$ satisfying $pp=p$. The set of idempotents of $S$ carries an order relation: 
$$p\leq q\quad \mbox{if}\quad p = pq = qp.$$ 
An idempotent is called {\em minimal} if it is minimal w.r.t.\ the above order.

\subsection{Inverses and completely regular elements of a semigroup}
An element $a$ of monoid $S$ is invertible if there exists $b\in S$ such that $ab=ba=1$, where $1$ is the unit element of $S$. $b$ is then called the inverse of $a$ and written $b = a^{-1}$. The invertible elements of a monoid form a group with neutral element $1$.
If the kernel of a monoid  contains an invertible element, then it contains the identity and hence coincides with the monoid. 
 
\begin{definition}
We call a monoid $S$ {\em nearly (left, right, or two-sided) simple} if $S$ has a unique minimal 
(left, right, or two-sided) ideal and that ideal contains all non-invertible elements. 
\end{definition}

More generally, in any semigroup $S$, $b\in S$ is called a {\em generalised inverse} of $a\in S$ if
$a= aba$ and $b = bab$. Not every element has a generalised inverse, neither are they unique when they exist. An element which admits a generalized inverse is called {\em regular}. It turns out that $a\in S$ is regular already if there exists $x\in S$ such that $a = axa$, because then
$y = xax$ is a generalized inverse for $a$.

An element $a\in S$ is called {\em completely regular} if there exists $x\in S: a = axa$ and $ax = xa$. This then implies that $y=xax$ is a generalised inverse for $a$ such that $ay=ya$.
A generalised inverse for $a$ which commutes with $a$ is unique if it exists.
We call such a commuting inverse the {\em normal inverse} of $a$ and denote it by $a^{-1}$, and set $a^0 = a a^{-1}=a^{-1}a$. 
An invertible element in a monoid is thus a completely regular element for which $a^0 = 1$, and there is no danger of confusion, as for such elements generalised inverses are unique and coincide with the monoid inverse.

Completely regular elements will play an important role in what follows. We provide two 
criteria for complete regularity. 
\begin{lemma}[\cite{howie1995fundamentals}]
Let $a\in S$. The following are equivalent:
\begin{enumerate}
\item $a$ is completely regular.
\item The $\Hh$-class of $a$, $\Hh_a$ is a group.
\end{enumerate}
\end{lemma}
The element $a^0$ is an idempotent and plays the r\^ole of the neutral element in the group  $\Hh_a$. So the normal inverse of $a$ in $S$ is the group inverse of $a$ in  $\Hh_a$.

We are particularily interested in subsemigroups of $X^X$, as we denote the semigroup of functions $f:X\to X$ with composition as semigroup product. In this case we have the following useful criterion for complete regularity.
\begin{lemma}\label{lem-cpreg1}
Let $f\in X^X$. 
$f$ is completely regular  if and only if its restriction to its image is bijective.
In this case $\im f = \im f^2$.
\end{lemma}
\begin{proof}
Suppose that $f$ is completely regular with normal inverse $g$. Then $f=fgf=f^2g=g f^2$.
This shows that $gf$ and $fg$ must both restrict to the identity on $\im f$. 
Hence the restriction of $f$ to its image is bijective.

Suppose that the restriction of $f$ to its image is bijective. Denote by $\tilde f:\im f\to \im f$ this restriction. Then $g = \tilde f^{-2} f$ is a normal inverse of $f$. Indeed $fg =  \tilde f^{-1} f$, $gf = \tilde f^{-2} f^2 = \tilde f^{-1} f$, $fgf =  f\tilde f^{-2} f^2 = f$, and $gfg =  \tilde f^{-2} f^2  \tilde f^{-2} f =  \tilde f^{-2} f$.

Clearly, if the restriction of $f$ to its image is bijective then $\im f = \im f^2$.
\end{proof}

\begin{definition}
A semigroup is called completely regular if all its elements are completely regular.
\end{definition}
\begin{lemma}\cite{howie1995fundamentals}\label{lem-union}
Let $S$ be a semigroup. 
The following assertions are equivalent.
\begin{enumerate}
\item $S$ is completely regular.
\item $S$ is a disjoint union of groups.
\item $S$ is a union of groups.
\end{enumerate}
\end{lemma}
As one might expect, the partition of $S$ into groups coincides with its partition into $\Hh$-classes. To describe a completely regular semigroup one needs, of course, not only to exhibit its groups, but also how elements of different groups multiply. The corresponding structure theory of completely regular semigroups is very rich \cite{petrich1999completely}.  

\begin{cor}\label{cor-union}
If $S$ is the union of completely regular sub-semigroups, then it is itself completely regular.
\end{cor}
\begin{proof} Clear from the description of completely regular semigroups as unions of groups.
\end{proof}

\begin{lemma}\label{lem-quot}
A surjective semigroup morphism $f:S\to S'$ preserves complete regularity. 
In particular, if $S$ is completely regular then also $S'$ is completely regular.
\end{lemma}
\begin{proof}
$f$ preserves the algebraic properties defining the normal inverse of an element.
 If $a^{-1}$ is a normal inverse of $a$ then $f(a^{-1})$ is a normal inverse of $f(a)$. 
 \end{proof}

\subsection{Simplicity and matrix semigroups}
Let $G$ be a group, let  $I$ and $\Lambda$ be  non-empty sets, and  let   
$A = (a_{\lambda i})_{i\in I,\lambda\in \Lambda}$ be a  $\Lambda\times I$ matrix with entries from $G$. Then the  {\em matrix semigroup} $M[G;I,\Lambda;A]$ 
 is the set $I\times G \times  \Lambda$ together with the multiplication
\begin{equation} \label{Rees-form}(i,g,\lambda)(j,h,\mu) = (i, g a_{\lambda  j} h,\mu).\end{equation}
The idempotents of $M[G;I,\Lambda;A]$ 
are $ (i,a_{\lambda i}^{-1},\lambda)$, $(i,\lambda)\in I\times \Lambda$, and they are all minimal.  The minimal left ideals are the sets $I\times G\times\{\lambda\}$
and thus in one to one correspondance with $\Lambda$ and its minimal right ideals are of the form
$\{i\}\times G\times \Lambda$
and thus in one to one correspondance with $I$. In particular, $M[G;I,\Lambda;A]$ is the union of its minimal left ideals, and the union of its minimal right ideals. 
Furthermore, the sets $\{i\}\times G\times \{\lambda\}$, $i\in I$, $\lambda\in\Lambda$ are subsemigroups which are groups with neutral element $(i,a_{\lambda i}^{-1},\lambda)$. 
Each of these groups is isomorphic to $G$ via 
$\{i\}\times G\times \{\lambda\}\ni (i,g,\lambda)\mapsto 
a_{\lambda i}g \in G$. $M[G;I,\Lambda;A]$ is thus a union of isomorphic groups.

There are natural bijections between the idempotents of distinct minimal left ideals, namely $(i,a_{\lambda i}^{-1},\lambda) \leftrightarrow (i,a_{\lambda' i}^{-1},\lambda')$.
This is usually formulated as follows: $p\leftrightarrow q$ iff $pq = q$ and $qp=p$. It is the restriction of the Green's relation $\mathcal R$ to the minimal idempotents.

 A {\em completely simple} semigroup (without zero element\footnote{Recall that no semigroup in this article will contain an element $0$ such that $0a=a0=0$ for all $a$.}) 
is a simple semigroup which contains a minimal 
idempotent. 
Completely simple semigroups are characterized by the following structure theorem.
\begin{theorem}[Rees-Suskevitch \cite{howie1995fundamentals}]\label{Rees-theorem}
A semigroup 
is completely simple if and only if it is isomorphic to a matrix semigroup $M[G;I,\Lambda;A]$.
\end{theorem}
Whereas the choice of $I$ and $\Lambda$ are canonical, namely $I$ is the set of $\Rr$-classes and $\Lambda$ the set of $\Ll$-classes of the semigroup, and for $G$ we may take the $\Hh$-class of a minimal idempotent (they are all isomorphic), the choice of the matrix
$A$ has some arbitrariness. 
But $A$ can be normalised
in such a way that the entries of one of its rows and one of its columns are all equal to     
the neutral element of $G$.

\begin{cor}\label{cor-ns}
Nearly simple monoids are completely regular if their kernel contains a minimal idempotent. 
\end{cor}
\begin{proof}
A nearly simple monoid is the union of the group of its invertible elements with its kernel.
If the kernel contains a minimal idempotent then, by Theorem~\ref{Rees-theorem}, it is a unions of groups. 
Now the claim follows from Lemma~\ref{lem-union}.
\end{proof}

A further consequence of the structure theorem is that any 
completely 
simple monoid is a group. To see this we first note that $\Hh_1$ is a group, as $1$ is completely regular. 
Now if $S$ is completely simple then, by the structure theorem, all idempotents of $S$ are minimal. Hence $1$ is a minimal idempotent. 
Since $1$ lies above all idempotents it is the only idempotent in $S$ and hence $S=\Hh_1$.

We later need the following result.
\begin{lemma}\label{lem-Rees-left}
Let $S$ be a {completely } simple semigroup which is the union $S=S_1\cup S_2$ of two left simple subsemigroups $S_1$, $S_2$. Then either $S$ is left simple or the union is disjoint and $S_1$ and $S_2$ are the minimal left ideals of $S$.
\end{lemma}

\begin{proof}
Let $S'$ be a left-simple subsemigroup of $S=M[G;I,\Lambda;A]$. Then either $S'\cap  I\times G\times\{\lambda\}$ is empty, or it is a left ideal of $S'$. Since $S'$ is left simple there must therefore exist a $\lambda\in\Lambda$ such that $S'\subset I\times G\times \{\lambda\}$. So by our assumptions there are $\lambda_1$ and $\lambda_2\in \Lambda$ such that $\Lambda=\{\lambda_1,\lambda_2\}$ and $S_k\subset I\times G\times \{\lambda_k\}$. If $\lambda_1=\lambda_2$ then $\Lambda$ is a single point and hence $S$ is left simple.  If $\lambda_1\neq\lambda_2$ then since $S=S_1\cup S_2$ we must have $S_k = I\times G\times \{\lambda_k\}$ and so $S_1$ and $S_2$ are the minimal left ideals of $S$. 
\end{proof}

\section{Ellis semigroups and complete regularity}

\subsection{Ellis semigroup of a dynamical system}
Let $X$ be a compact metrizable space. 
We equip $X^X$, the semigroup of functions $X\to X$,
 with the topology of pointwise convergence, this is the same as the product topology. Right multiplication with an element $f\in X^X$,
$\rho_f:X^X\to X^X$, $\rho_f(g) = gf$ is continuous, but, if $X$ is infinite, not left multiplication. A semigroup with these properties is called right-topological.

Let $T$ be a semigroup with an action $\alpha$ on $X$ by continuous surjective maps.
We refer to such a triple $(X,\alpha,T)$ as a dynamical system.
For each $t\in T$, $\alpha^t$ is an element of  $X^X$ and\footnote{We use additive notation as we are mainly interested in abelian $T$.} 
$\alpha^{s+t} = \alpha^s\circ\alpha^t$. We suppose that $T$ has an identity element $0$ so that $\alpha^0=\id$ is the identity map on $X$, which we also denote by $1$.
The Ellis semigroup $E(X,T)$ is the (topological) closure of $\{\alpha^t:t\in T\}$ in $X^X$. 
This closure is also closed under composition of functions. $E(X,T)$ is thus a compact right-topological monoid.

We summarize the implications of compactness of a right topological semigroup which we need in the following theorem, which can for instance be found in \cite{hindman}. 
\begin{theorem}\label{thm-Rees-top}
A compact right topological semigroup $S$ admits a kernel $\M$ and this kernel contains all minimal idempotents. $\M$ is the disjoint union of the minimal left ideals of $S$. 
Each minimal left ideal is compact and contains an idempotent.
\end{theorem}
Hence Ellis semigroups are monoids which admit a kernel, and this kernel is isomorphic to a matrix semigroup. Combining this with Cor.~\ref{cor-ns} we see that a nearly simple Ellis semigroup is completely regular. We denote the kernel of $E(X,T)$ by $\M(X,T)$. 

If the identity is the only idempotent of $E(X,T)$ then $E(X,T)$ is a group,
because by the above $\id$ is then in the kernel and a completely simple semigroup with a single idempotent is a group. 
Thus $E(X,T)$ is a group if and only if it contains besides $\id$ no other idempotent.

We recall three fundamental results. Proofs for the case $T=\Z$ can be found in \cite{Auslander}. For the benefit of the reader carry them out for general $T$. While the proof of the first result comes close to a topological tautology, the other proofs are purely algebraic and more or less direct consequences of the Rees structure theorem. 
\begin{theorem} \label{thm-fund}
Let $(X,T)$ be a dynamical system.
\begin{enumerate}
\item Two points $x,y\in X$ are proximal if and only if there is a function $f\in E(X,T)$ such that $f(x)=f(y)$. In particular, for any idempotent $p$, $p(x)$ is proximal to $x$.
\item Proximality is a transitive relation if and only if $E(X,T)$ has a unique minimal left ideal. 
\item
$(X,T)$ is distal if and only if $E(X,T)$ is a group.
\end{enumerate}
\end{theorem}
\begin{proof}${}$
(1) Let $\inf_{t\in T} d(\alpha^t(x),\alpha^t(y))=0$. Then there is a net $(t_\nu)$ such that 
$\lim_{\nu} d(\alpha^{t_\nu} (x),\alpha^{t_\nu}(y))=0$. By compactness $(\alpha^{t_\nu})$ admits a converging subnet whose limit is a function $f\in E(X,T)$. It follows from the continuity of $d$ that $d(f(x),f(y))=0$. 
As for the converse, any $f\in E(X,T)$ is the limit of a net $(\alpha^{t_\nu})$ which converges pointwise. By continuity of the metric we then have $\inf_{t\in T} d(\alpha^t(x),\alpha^t(y))\leq 
d(f(x),f(y))$ so that $f(x)=f(y)$ implies $\inf_{t\in T} d(\alpha^t(x),\alpha^t(y))=0$.

(2) Let $x$ and $ y$ be proximal. Then there exists $f\in E(X,T)$ such that $f(x)=f(y)$. 
All $g\in L:=\M(X,T) f$ satisfy $g(x)=g(y)$. $L$ 
is a left ideal of the kernel $\M(X,T)$ and hence, by Thm.~\ref{thm-Rees-top}, a minimal left ideal of $E$. Thus $x$ and $ y$ are proximal if and only if there is a minimal left ideal $L$ such that $f(x)=f(y)$ for all $f\in L$. 

Suppose there is a unique minimal left ideal $L$. Then $x$ being proximal to $y$ is equivalent to
$f(x)=f(y)$ for all $f\in L$. The latter defines, of course, a transitive relation.

If proximality is transitive then, given $x\in X$, and any two idempotents $p$, $q$, 
$p(x)$ and $q(x)$, being each proximal to x, are proximal to each other. Hence, given $x$, there is a minimal left ideal $L$ such that for all its idempotents $r\in L$ we have $rp(x) =rq(x)$. 
Now suppose that $p$ and $q$ belong to the same  minimal right ideal and choose $r$ to be in that same minimal right ideal. This is possible by the Rees structure theorem as any minimal left ideal intersects any minimal right ideal. 
Then we get $p(x) = rp(x) = rq(x) = q(x)$ (by product rule of idempotents in minimal right ideals). 
We thus have shown that, for all $x$ and any pair of minimal idempotents in the same right ideal, $p(x)=q(x)$. Hence minimal right ideals have a single idempotent. Hence there is only one minimal left ideal.

(3)
Suppose that proximality is non-trivial in the sense that there are $x\neq y$ which are proximal. Then there is $f\in E(X,T)$ which has the same image on $x$ and $y$ and hence is not injective. Hence $f$ cannot be invertible and $E(X,T)$ is not a group. As for the converse, suppose that $E(X,T)$ is not a group and thus contains an idempotent $p\neq \id$.
Then there is $x\in X$ with $p(x) \neq x$. As $p(x)$ and $x$ are proximal, proximality is non-trivial.
\end{proof}
It should be kept in mind that when $E(X,T)$ is a group this means automatically that it is a group of bijections of $X$. On the other hand, when we consider subsemigroups of $E(X,T)$ which are groups, then these need not to contain $\id$, but their neutral element only needs to be an idempotent $p$ of $E(X,T)$. This subgroup will then not contain bijections of $X$ but only maps whose restrictions on $p(X)$ are bijective. A non-invertible function can thus still belong to a subgroup of $E(X,T)$.

\subsection{Criteria for complete regularity}

\begin{cor} Let $(X,\alpha,T)$ be a dynamical system and 
$T = T_1\cup \cdots\cup T_k$ a decomposition into a finite union of subsemigroups. 
If all $E(X,T_i)$ 
are completely regular then $E(X,T)$ is completely regular.
\end{cor}
\begin{proof} Since the closure of a finite union of subsets of a topological space is the union of their closures
the statement follows from Cor.~\ref{cor-union}.
\end{proof}
A continuous surjection $ \pi:X \to Y$ from a dynamical system $(X,\alpha,T)$ to a dynamical system  $(Y,\beta,T)$ which is equivariant w.r.t.\ the actions, $\pi\circ\alpha^t=\beta^t\circ\pi$, is called a {\em factor map}. One simply says that $(Y,\beta,T)$ is a factor of $(X,\alpha,T)$. 
A factor map $\pi:X\to Y$ induces a surjective morphism of semigroups $\pi_*:E(X,T)\to E(Y,T)$, namely $\pi_*(f)(y) = \pi(f(x))$ where $x$ is a preimage of $y$ under $\pi$ \cite{Auslander,hindman}. The following corollary is thus an immediate consequence of Lemma~\ref{lem-quot}.

\begin{cor}\label{cor-factor}
Suppose that $(Y,\beta,T)$ is a factor of $(X,\alpha,T)$. If  $E(X,T)$ is completely regular then also $E(Y,T)$ is completely regular.
\end{cor}

\subsection{$\Z$ and $\R$-actions}
In this section we will focuss on actions of $\Z$ and $\R$ by homeomorphisms. They can be decomposed into their forward and their backward actions. So $T$ will denote $\Z$ or $\R$, and $T^+$  and $T^-$ their subsemigroups $\Z^+$ and $\Z^-$, or $\R^+$ and $\R^-$. 
Note that the results of the last section apply when inserting $T^+$ or $T^-$ for $T$. 
A fruitful question to ask is what we can say about $E(X,T)$ knowing the structure of $E(X,T^+)$ and $E(X,T^-)$. Questions of this type have also been analysed in \cite{AkinAuslanderGlasner}[Theorem~8.1] for $T=\Z$ using the formalism of Ellis actions.  
To simplify the notation we use also the following abreviations 
$E=E(X,T)$, $\M=\M(X,T)$ and $T=\{\alpha^t | t\in T\}$ and denote by ${}^\pm$ the restriction to the $T^\pm$-action. In particular, $E^+ = E(X,T^+)$ and $\M^+=\M(X,T^+)$. 

\begin{lemma}\label{lem-E1} Let $(X,\alpha,T)$ be a dynamical system with $T=\Z$ or $\R$.
Let $f\in E$, $g\in E^+\backslash T^+$. Then $fg\in E^+$.
\end{lemma}

\begin{proof}
Let $f\in E$, $g\in E^+\backslash T^+$. So $f = \lim \sigma^{n_\nu}$ and $g = \lim \sigma^{m_\mu}$, however with $m_\mu\to+\infty$. Then $fg = \lim_\nu \sigma^{n_\nu} g$.
Since $\sigma^{n_\nu} g = \lim_{\mu}  \sigma^{n_\nu+m_\mu} \in E^+$ and $E^+$ is closed 
we have $fg\in E^+$.
\end{proof}
\begin{prop}\label{prop-plus}
Any minimal left ideal $L^+$ of $E^+$ is a minimal left ideal of $E$. In particular, if $E^+$ is a group then $E=E^+$.
\end{prop}
\begin{proof}
Let $L^+$ be a minimal left ideal of $E^+$.
Suppose first that $E^+$ is a group which means that $L^+$ does not contain an idempotent different from $1$. Then $L^+$ equals $E^+$ and is a group. As $T^+\subset X^X$ is not closed there is an element $f\in E^+\backslash T^+$.  By Lemma~\ref{lem-E1}, $Ef \subset E^+$. But $f$ is invertible in $E^+$ and since $E^+\subset E$, $E$ contains the function $f^{-1}$ which is inverse to $f$. Hence $Ef = E$. Thus we have the chain of semigroup inclusions $E\subset E^+\subset E$ showing that $E^+=E$. 

Now consider the case in which $E^+$ is not a group so that $L^+$ contains an idempotent $p$ which is strictly smaller than $1$.  Then $p\in L^+\backslash T^+$ and hence, by Lemma~\ref{lem-E1}, $Ep = Ep p\subset E^+ p = L^+$.  
Hence $L^+$ is a left ideal of $E$.
To show that it is a minimal left ideal let $q$ be an idempotent smaller than $p$. Then $q = qp$ which implies, by the above that $q\in L^+$. As all idempotents of $L^+$ are minimal we have $q=p$. This shows that $p$ is also minimal in $E$. Hence $L^+=Ep$ 
is a minimal left ideal of $E$. 
\end{proof} 
The following corollary is part of a list of results of \cite{AkinAuslanderGlasner}[Theorem~8.1].
\begin{cor}
$(X,\alpha,T^+)$ is distal if and only if $(X,\alpha,T)$ is distal if and only if  $(X,\alpha,T^-)$ is distal.
\end{cor}
\begin{proof}
If $(X,\alpha,T^+)$ is distal then $E^+$ a group. By Prop.~\ref{prop-plus} this implies $E^+=E$ and hence $(X,\alpha,T)$ is distal. This, in turn, means that $(X,\alpha,T^+)$ and $(X,\alpha,T^-)$ are distal. 
\end{proof}

\begin{cor}\label{cor-E1} 
Let $(X,\alpha,T)$ be a dynamical system with $T=\Z$ or $\R$. We have
$$\M=\M^+\cup \M^-$$ and $\M^+$ and $\M^-$ are left ideals of $E$. 
\end{cor}
\begin{proof}
If $(X,\alpha,T)$ is distal then all the above kernels are equal to the group $E$ so the statement is evident. Otherwise $1$ is not a minimal idempotent and not 
in $\M^+$ or $\M^-$. As both are unions of minimal ideals (for $E^+$ and $E^-$) Prop.~\ref{prop-plus} shows that they are left ideals of $E$. Moreover, we saw that if $p$ is an idempotent of $\M^+$ then $Ep$ is a minimal left ideal of $E$. As any idempotent must belong to $E^+$ or to $E^-$ and idempotents outside of $\M^+\cup \M^-$ cannot be minimal in $E$, we find that the union $\M^+\cup \M^-$ exhausts all minimal left ideals of $E$. Hence the union must be $\M$.
%
\end{proof}

Recall that a pair $x,y\in X$ is called proximal if 
$\inf_{t\in T}d(\alpha^t(x),\alpha^t(y)) = 0,$
and forward proximal if 
$$\inf_{t\in T^+}d(\alpha^t(x),\alpha^t(y)) = 0.$$
Replacing $T^+$ by $T^-$ we obtain the corresponding notion of backward proximality.
Note that if proximality is a transitive relation, then also forward and backward proximality are transitive, but the converse need not be true. 

We let $J_{min}$ and $J^\pm_{min}$ denote all minimal idempotents of $E$ and $E^\pm$, respectively. 

\begin{theorem}\label{thm-E1} 
Let $(X,\alpha,T)$ be a dynamical system with $T=\Z$ or $\R$. Suppose that forward proximality and backward proximality are transitive. 
We have the following dichotomy: 
\begin{enumerate}
\item Either one of the following equivalent statements holds
\begin{itemize}
\item[(i)] Proximality is transitive, 
\item[(ii)]  $E$ has a unique minimal left ideal, 
\item[(iii)] $\M=\M^+=\M^-$, 
\item[(iv)] $J^+_{min} = J^-_{min}$. 
\end{itemize}
\item or one of the following equivalent statements holds
\begin{itemize}
\item[(i)] Proximality is not transitive, 
\item[(ii)] $E$ has exactly two minimals left ideals namely $\M^+$ and $\M^-$, 
\item[(iii)]  $\M^+\cap\M^-=\emptyset$
\item[(iv)] $J^+_{min}\cap J^-_{min}=\emptyset$.
\end{itemize}
\end{enumerate}
\end{theorem}
\begin{proof}
By Cor.~\ref{cor-E1}, $\M=\M^+\cup \M^-$ and $\M^+$ and $\M^-$ are left ideals of $E$ hence of $\M$. As both, $\M^+$ and $\M^-$ are left simple by assumption (see Thm.~\ref{thm-fund}) we can  
apply Lemma~\ref{lem-Rees-left} to see that either $\M$ is the disjoint union of 
$\M^-$ with $\M^+$ or $\M$ is left simple. 

Suppose that $\M$ is left simple. By Thm.~\ref{thm-fund} this is equivalent to proximality being transitive. Since $\M^+$ and $\M^-$ are left ideals of $\M$ (by Cor.~\ref{cor-E1}) they must be equal. Therefore also $J^+_{min} = J^-_{min}$.  
We saw above that for $p\in J^+_{min}$ we have $Ep\subset \M^+$. Hence $J^+_{min} = J^-_{min}$ implies $\M^+=\M^-$ which implies that $\M$ is left simple.

Suppose that $\M$ is not left simple. By Lemma~\ref{lem-Rees-left} this is equivalent to $\M^+\cap\M^-=\emptyset$ and implies $J^+_{min}\cap J^-_{min}=\emptyset$. On the other hand, if $p\in J^+_{min}\cap J^-_{min}$ then $Ep\subset \M^+\cap\M^-$, so $\M^+\cap\M^-\neq\emptyset$.
\end{proof}

We recall from the discussion after Theorem~\ref{Rees-theorem} that in the last case, where $E$ has two minimal left ideals, there is a canonical one-to-one correspondence between $J^+_{min}$ and $J^-_{min}$:  
exactly one idempotent from $J^+_{min}$ is $\Rr$-related to exactly one idempotent of 
$J^-_{min}$. This can also be found in \cite{Auslander}.

\section{Almost distal systems}
The notion of almost distal systems was introduced in \cite{Blanchard} for $\Z^+$-actions. We investigate it here for $\Z$ or $\R$ actions. 
In this section $T$ will again denote $\Z$ or $\R$. 
\subsection{Almost distal $T^+$-actions}
\begin{definition} Let $(X,\alpha,T^+)$ be a dynamical system with an action by continuous surjectives maps.
A pair $x,y\in X$ is called proximal if 
$$\inf_{t\in T^+}d(\alpha^t(x),\alpha^t(y)) = 0$$
and asymptotic if 
$$\lim_{t\to +\infty}d(\alpha^t(x),\alpha^t(y)) = 0.$$
$(X,\alpha,T^+)$ is called {\em distal} if proximality implies equality.
$(X,\alpha,T^+)$ is called {\em almost distal} if proximality implies asymptoticity. 
\end{definition}
Note that the asymptoticity relation is always transitive. 

Almost distal $\Z^+$-actions have been studied in \cite{Blanchard}  with the help of the so-called {\em adherence semigroup}. The adherence semigroup of $(X,\alpha,T^+)$ is the subsemigroup $\Aa(X,T^+)\subset E(X,T^+)$ of elements which are obtained as limits of nets $f=\lim \alpha^{t_\nu}$ where $\lim t_\nu = +\infty$. Equivalently, $\Aa(X,T^+)=\bigcap_{t\in T^+} \alpha^t E(X,T^+)$, so it is a closed set. We also denote $\Aa^+=\Aa(X,T^+)$.

\begin{lemma} \label{cor-1}
Consider a dynamical system $(X,\alpha,T^+)$.
If $f\in E^+\backslash T^+$ then 
$f=\lim \alpha^{t_\nu}$ for a net $(t_\nu)$ for which $\lim t_\nu = +\infty$. In particular, 
$E^+\backslash T^+\subset \Aa^+$.
\end{lemma}
\begin{proof}
We have $f=\lim \alpha^{t_\nu}$ where $(t_\nu)$ is a net in $T^+$. It is as well a net in 
the one-point compactification $T^+\cup \{+\infty\}$. By compactness of the latter
$(t_\nu)$ has a subnet $(t'_\nu)$ which converges. Also  $(\alpha^{t'_\nu})$ converges to $f$.
If $\lim t'_\nu = t\in T^+$ then $f=\alpha^t \in T^+$. \end{proof}

\begin{lemma}
$\Aa^+$ contains  $\M^+$. 
\end{lemma}
\begin{proof}
Clearly, if $g=\alpha^t\in T^+$ and $f\in\Aa^+$ then $fg=gf\in\Aa^+.$ If $f,g\in\Aa^+$ then $fg,gf\in\Aa^+$ as $\Aa^+$ is closed under multiplication. Hence  $E^+\Aa^+\cap \Aa^+E^+\subset \Aa^+$, by Lemma~\ref{cor-1}, and $\Aa^+$, being a two-sided ideal, must contain the kernel of $E^+$. 
\end{proof}
\begin{lemma} \label{lem-2}
 If $\Aa^+$ is simple then $\Aa^+ = \M^+$ and $E^+$ is nearly simple. If $\Aa^+$ is left simple then $E^+$ is nearly left simple. 
\end{lemma}
\begin{proof}
$\M^+$ is an ideal of $\Aa^+$. Hence, if $\Aa^+$ is simple then $\Aa^+ = \M^+$ and, by Lemma~\ref{cor-1}, $E^+=\M^+\cup T^+$. Hence $E^+$ is nearly simple. 

A left simple semigroup is simple. Hence if $\Aa^+$ is left simple then $\Aa^+ = \M^+$ and $E^+=\M^+\cup T^+$, and $\M^+$ is left simple.                    
\end{proof}
The following theorem is proved in \cite{Blanchard} for $T^+=\Z^+$, and the proof given there carries over for $T^+=\R^+$.
\begin{theorem}\label{thm-Blanchard}
$(X,\alpha,T^+)$ is almost distal if and only if its adherence semigroup $\Aa(X,T^+)$ is left simple.
\end{theorem}
The proof is based on a simple lemma which we need explicitly.
\begin{lemma}\label{lem-E21} 
Let $(X,\alpha,T^+)$ be a dynamical system and $f\in \Aa(X,T^+)$. If $x,y$ are asymptotic then $f(x) = f(y)$. 
\end{lemma}
\begin{proof} By assumption $f=\lim \alpha^{t_\nu}$ with $\lim t_\nu = +\infty$.
Hence for any finite $t\in T^+$ there exists $\nu_0$ such that 
$t_{\nu}\geq t$ for all $\nu\succeq\nu_0$. In particular, if $x$ and $y$ are asymptotic points, so that $\lim_{t\to\infty} d(\alpha^t(x),\alpha^t(y))=0$, then $d(f(x) , f(y))=0$. 
\end{proof}
In a similar context, we recall the corollary from \cite{ESBS}. 
\begin{cor}\label{cor-2} If  $(X,\alpha,T^+)$ is almost distal then $E(X,T^+)$ is nearly left simple.
\end{cor}
\begin{proof} This follows from Thm.~\ref{thm-Blanchard} and Lemma~\ref{lem-2}. A direct proof which is based on the last lemma goes as follows: Given any idempotent $p\in E^+$, any $x\in X$ is proximal to $p(x)$, and hence, if $(X,\alpha,T^+)$ is almost distal, asymptotic to $p(x)$. Now Lemma~\ref{lem-E21} shows that $f(p(x))=f(x)$
provided $f\in \Aa^+$. Since $x$ was arbitrary we find $f=fp$. Hence any  $f\in \Aa^+$ lies in the left ideal generated by the idempotent $p$. If $p$ is minimal then this left ideal is a minimal left ideal. Since $p$ can be any minimal idempotent there can only be one minimal left ideal. It follows that
$E^+ = \Aa^+ \cup T^+\subset  E^+p \cup T^+= \M^+\cup T^+$.
Hence $E^+ =  \M^+\cup T^+$ and $\M^+$ is left simple.
\end{proof}

\subsection{Almost distal $T$-actions}
We now consider a system $(X,\alpha,T)$ with an action by homeomorphisms together with its restrictions to the forward and to the backward motion which are $(X,\alpha,T^+)$ and $(X,\alpha,T^-)=(X,\alpha^{-1},T^+)$. Note that the $T^+$ and the $T^-$-action are then not only surjective but even bijective. In this case forward proximality (or asymptoticity) is the same as proximality (or asymptoticity) for the $T^+$-action. 
\begin{definition}
Let $(X,\alpha,T)$ be a dynamical system with an action by homeomorphisms. It is called {\em almost distal}, if both, $(X,\alpha,T^+)$ and $(X,\alpha,T^-)$ are almost distal.
\end{definition}
Note that for an almost distal system, forward and backward proximality are transitive relations.

\begin{prop} \label{prop-almost-distal}
Let $(X,\alpha,T)$ be an almost distal system. Then $E$ is nearly simple. 
\end{prop}
\begin{proof} 
By Cor.~\ref{cor-E1} we have $\M = \M^+\cup\M^-$. As $E=E^+\cup E^-=\M^+\cup\M^-\cup T$ we see that all non-invertible elements belong to $\M$.
\end{proof}

As an application we consider 
subshifts defined by bijective substitutions. Such a subshift is almost distal \cite{ESBS}, so by Prop.~\ref{prop-almost-distal} it has a nearly simple Ellis semigroup. Furthermore, the proximality relation for such a subshift is not transitive, as its coincidence rank is larger than $1$ \cite{Aujogue-Barge-Kellendonk-Lenz}. By Theorem~\ref{thm-E1} its Ellis semigroup has exactly two minimal left ideals, one associated with the forward and one with the backward motion.

\subsection{$\Z^+$ and $\Z$-actions on totally disconnected spaces} 

A topological space is {\em totally disconnected} if it has a base of clopen subsets. Well known examples are subshift spaces. The one-sided, or two-sided, full shift over a finite alphabet $\Aa$ is the space of sequences $\Aa^{\Z^+}$, or $\Aa^\Z$, equipped with the product topology. This topology is metrisable; we may for instance use the metric $d(x,y) = e^{-N(x,y)}$ where $N(x,y)$ is the supremum of all $N$  such that for all $|n|\leq N$ we have $x_n=y_n$. The closed ball of radius $e^{-N}$ centered at $x$ is the set of sequences $y$ which agree with $x$ for all indices $|n|\leq N$. Its complement is a finite union of such balls, so closed balls are open and  $\Aa^{\Z^+}$ and $\Aa^\Z$  totally disconnected.
We denote the (left) shift on  $\Aa^{\Z^+}$ and $\Aa^\Z$ by $\sigma$: $\sigma(x)_n = x_{n+1}$. It is a continuous surjective map on $\Aa^{\Z^+}$ and a 
homeomorphism on $\Aa^\Z$. 
The restriction of $\sigma$ to 
any closed shift invariant subspace of $\Aa^{\Z^+}$, or $\Aa^\Z$, is 
a topological dynamical system on a totally disconnected space, it is called a one-sided, or two-sided, subshift.

On a subshift $(X,\sigma)$, two sequences $x,y\in X$ are forward asymptotic if and only if they agree to the right, that is, there exists $n_0$ such that $x_n=y_n$ for all $n> n_0$. They are forward proximal, precisely if they agree on larger and larger segments on the right, that is, for all $N$ exists $n_0\geq 0$ such that $x_n=y_n$ for all $n_0< n\leq N+n_0$.
 
Let $(X,\alpha)$ be a $\Z^+$, or a $\Z$-action on a compact space $X$ and $U_1,\cdots, U_k$ a partition of $X$ into clopen subsets. Let $\Aa=\{1,\cdots,k\}$ viewed as an alphabet of $k$ letters.
The coding of $(X,\alpha)$ defined by the partition is the map $\phi:X\to\Aa^{\Z^+}$, or $\phi:X\to\Aa^\Z$, given by $\phi(x)_n = i$ if $\alpha^n(x)\in U_i$. By construction, $\phi$ is a continuous $\Z$-equivariant map and so its image is a compact shift invariant subspace, that is, a subshift. Hence $\phi:(X,\alpha)\to (\phi(X),\sigma)$
is a factor map onto a subshift.
\subsection{Li-Yorke pairs}
A forward Li-Yorke pair is a forward proximal pair which is not forward asymptotic. (For $\Z$-actions, forward means for the $\Z^+$-action). 
Note that in a subshift space, a pair $x,y$ is Li-Yorke if and only if there exists two strictly increasing sequences $(n_k)_k$, $(N_k)_k$ of $\Z^+$ such that 
\begin{equation}\label{eq-LY-pair}
x_{n_k}\neq y_{n_k}\quad \mbox{but}\quad \forall n_k<n\leq N_k+n_k : x_n = y_n
\end{equation}

\begin{lemma}\label{lem-sub}
Consider a (one- or two-sided) subshift $(X,\sigma)$
which has a forward Li-Yorke pair $x,y$. There exists $f\in E^+$ 
such that $f(x)\neq f(y)$ but $f(x)$ and $f(y)$ are forward asymptotic. In particular, $f$ is not injective on its image and therefore not completely regular.

\end{lemma}
\begin{proof}
Given a forward Li-Yorke pair $x,y$ let  $(n_k)_k$, $(N_k)_k$ be strictly increasing sequences satisfying (\ref{eq-LY-pair}). By compactness of $X$ we may go over to subsequences to assure that 
$\sigma^{n_k}(x)$ and $\sigma^{n_k}(y)$ converge towards
 $\tilde x$ and $\tilde y$.
These satisfy $\tilde x_0\neq \tilde y_0$ and $\tilde x_n=\tilde y_n$ for all $n>0$. 
Hence $\tilde x$ and $\tilde y$ are forward asymptotic but not equal. By compactness of $E(\Z^+)$ there exists $f\in E(\Z^+) $ such that $f(x) = \lim_k\sigma^{n_k}(x)=\tilde x$ and $f(y) = \lim_k\sigma^{n_k}(y)=\tilde y$. By Lemma~\ref{lem-E21}, $f(\tilde x) =f(\tilde y) $, as   $\tilde x$ and $\tilde y$ are forward asymptotic. Hence $f$ is not injective on its image. By Lemma~\ref{lem-cpreg1}, $f$ is not completely regular.
\end{proof}

\begin{prop} \label{prop-Li-Y}
The Ellis semigroup $E(X,\Z^+)$ of a dynamical system $(X,\alpha,\Z^+)$ 
on a totally disconnected compact metric space 
which admits a Li-Yorke pair is not completely regular.
\end{prop}
\begin{proof} Choose a metric $d$.
Let $l,y$ be a forward Li-Yorke pair so that 
$$\delta := \limsup_{n\to+\infty} d(\alpha^n(l),\alpha^n(y))>0.$$ Consider a partition $\{U_1,\cdots,U_k\}$ of $X$ into clopen subsets of size $\frac{\delta}2$. Coding with this particion yields a factor map $\phi$ onto a subshift. 
Factor maps preserve asymptoticity and, if 
$d(\alpha^n(l),\alpha^n(y))\geq \delta$ then $\phi(l)_n\neq \phi(y)_n$. Thus $ \phi(l),\phi(y)$ is a forward Li-Yorke pair of the subshift. By Lemma~\ref{lem-sub} the Ellis semigroup of the subshift is not completely regular. By Cor.~\ref{cor-factor} $E(X,\Z^+)$ is not completely regular.
\end{proof}
This leads to the main theorem of our work.
\begin{theorem} \label{thm-main}\label{thm-complete-regularity}
Consider a dynamical system $(X,\alpha,T)$ on a totally disconnected compact metric space where $T=\Z^+$ or $T=\Z$. The following assertions are equivalent.
\begin{enumerate}
\item\label{item1} $(X,\alpha,T)$ is almost distal.
\item\label{item2} $E(X,T)$ is nearly simple.
\item\label{item3} $E(X,T)$ is completely regular.
\end{enumerate}
\end{theorem}
\begin{proof}
Let $T=\Z^+$. The implication (\ref{item1}) $\Rightarrow$ (\ref{item2}) is Cor~\ref{cor-2}. 
The implication (\ref{item2}) $\Rightarrow$ (\ref{item3}) follows from Cor~\ref{cor-ns}, as we have already seen. Finally, the  implication (\ref{item3}) $\Rightarrow$ (\ref{item1}) can be seen as follows: 
A system which is not almost distal must contain a Li-Yorke pair and thus, by Prop.~\ref{prop-Li-Y}, cannot be completely regular.

Let $T=\Z$. The implication (\ref{item1}) $\Rightarrow$ (\ref{item2}) is Prop.~\ref{prop-almost-distal}. The implication (\ref{item2}) $\Rightarrow$ (\ref{item3}) follows again from Cor~\ref{cor-ns}, and (\ref{item3}) $\Rightarrow$ (\ref{item1}) follows again from Prop.~\ref{prop-Li-Y} applied to the forward and the backward motion.
\end{proof}

\newcommand{\Zl}{{ \Z_\ell}}
\newcommand{\Ef}{E^{fib}}

\newcommand{\al}{{\color{orange}\alpha}}
\newcommand{\be}{{\color{orange}\omega}}
\newcommand{\alb}{{\color{blue}\bar\alpha}}
\newcommand{\beb}{{\color{blue}\bar\omega}}
\newcommand{\eps}{\epsilon}
\renewcommand{\a}{{\color{blue}\boxdot}}
\renewcommand{\c}{{\color{red}\boxminus}}
\renewcommand{\b}{{\color{orange}\boxplus}}
\newcommand{\cb}{{\boxtimes}}
\newcommand{\ab}{{\Box}}

\renewcommand{\aa}{\a\a\c\a\a}
\newcommand{\bb}{\a\b\c\a\a}
\newcommand{\cc}{\a\c\c\b\a}
\newcommand{\aaa}{\aa\aa\cc\aa\aa}
\newcommand{\bbb}{\aa\aa\cc\bb\aa}
\newcommand{\ccc}{\aa\cc\cc\bb\aa}
\newcommand{\aaaa}{\aaa\aaa\ccc\aaa\aaa}
\newcommand{\bbbb}{\aaa\aaa\ccc\bbb\aaa}
\newcommand{\cccc}{\aaa\ccc\ccc\bbb\aaa}

\section{Example: a non-completely regular $E(X,\Z)$}
We provide an explicit example of a dynamical system which has an Ellis semigroup which is not completely regular. Incidently our example is backward almost distal but not forward almost distal. Examples which are only one-sided almost distal were already constructed in \cite{AkinAuslanderGlasner}. Our example is given by a simple substitution of constant length.

General background on constant length substitutions, their dynamical systems, the description of their maximal equicontinuous factor and its fibres can be found in \cite{dekking,coven-quas-yassawi,ESBS}. We will use these results freely.

The substitution 
we will look at is defined on three symbols $\Aa = \{\a,\b,\c\}$. 
It is given by the map $\theta:\Aa\to\Aa^5$ (we think of $\Aa^5$ as words of length $5$ in $\Aa$)
$$
\begin{matrix}
\a \mapsto \a\a\c\a\a\\
\b \mapsto \a\b\c\a\a\\
\c \mapsto \a\c\c\b\a\\
\end{matrix}
$$
Extending this substitution by concatenation one obtains arbitrarily long words upon iterating $\theta$ on one symbol. The substitution dynamical system defined by $\theta$ is the subshift $(X_\theta,\sigma)$ whose space $X_\theta\subset \Aa^\Z$ contains all those sequences whose subwords occurr in some $\theta^N(\a)$, $N\in\N$. The general theory provides us with the following information.
\begin{enumerate}
\item The maximal equicontinuous factor of $(X_\theta,\sigma,\Z)$ is the adding machine in base $5$, $(\Z_5,(+1))$. One can think of the elements of $\Z_5$ as one sided infinite sequences $\{0,1,2,3,4\}^\N$, and the action of $\Z^+$ is given by addition in base $5$ with carry to the right. $\Z_5$ is an abelian group with neutral element $\bar 0$, as we denote the infinite sequence of $0$s.
\item The factor map $\pi:X_\theta \to \Z_5$ is one-to-one except on the pre-images of one orbit (under addition {of 1}) of points in $\Z_5$, namely the orbit of $\bar 1$, as we denote the infinite sequence of $1$s; we denote this orbit by $O_{\bar 1}$. This can be easily computed 
following the algorithm given in Sect.~3.4 of \cite{coven-quas-yassawi}.
\item The elements of the fibre $\pi^{-1}(\bar 1)$ are the fixed points under the map
$\tilde \theta:=\sigma\circ \theta$. There are three of them, in bijection to $\Aa$, obtained from the seeds $\a,\b,\c$. We show two iterations:
\begin{equation}\label{eq-fp}
\begin{matrix}
 . \a\\
 . \b\\
 . \c
\end{matrix} \stackrel{\widetilde{\theta}}\to \begin{matrix}
 \a . \a \c\a\a \\
 \a.  \b \c\a\a \\ 
 \a . \c \c\b\a
\end{matrix} \stackrel{\widetilde{\theta}}\to \begin{matrix}
\aa  \a . \a \c\a\a  \cc\aa\aa \\
\aa  \a . \b \c\a\a  \cc\aa\aa \\
\aa  \a . \c \c\b\a  \cc\bb\aa
\end{matrix}
\end{equation}
Here the dot serves to position the words in a bi-infinite sequence $x\in\Aa^\Z$, namely the letter to the right of the dot is $x_0$. We denote by $x_{\sf}$ the fixed point sequence with seed $\sf$.
All elements of  $\pi^{-1}(\bar{1})$ agree to the left, two agree also to right but the third one agrees with the others to the right only on larger an larger patches. Since $\sigma$ corresponds to the left shift, there are no backward Li-Yorke pairs, but two forward Li-Yorke pairs. Nevertheless, proximality is transitive for the subshift.
\end{enumerate}
\subsection{The fiber preserving part of $E(X_\theta,\Z)$}
Let $\Ef=\Ef(X_\theta,\Z)$ be the set of elements of $E=E(X_\theta,\Z)$
which preserve the fibers $\pi^{-1}(z)$ of the maximal equicontinuous factor map. It is easily seen \cite{ESBS} that these are precisely the elements which lie in the kernel of the map $\pi_*:E(X_\theta,\Z)\to E(\Z_5,\Z)$ induced by the fibre map on the Ellis semigroups. As $(\Z_5,(+1))$ is distal (and minimal abelian), evaluation at $\bar 0$ yields a group isomorphism $ E(\Z_5,\Z)\to\Z_5$. Moreover, an element $f\in E$ which preserves one fibre will preserve all fibers. 

We denote by $f_z$ the restriction of $f\in E$ to a map
$$f_z:\pi^{-1}(z) \to \pi^{-1}(z+\eta)$$
where $\eta=\pi_*(f)(\bar 0)\in\Z_5$. In particular, $\eta=0$ if and only if $f\in \Ef$.
An element of $E$ can be described as follows.
\begin{enumerate}
\item If $f\in \Ef$ then $f_z=\id$ provided $z\notin O_{\bar 1}$ whereas we can view  $f_{\bar 1}$ as a map from $\Aa$ to $\Aa$. Furthermore,
$f_{z+1} = \sigma f_z \sigma^{-1}$ by equivariance of the factor map. 
Hence any element of $\Ef$ is completely determined by a map from $\Aa$ to $\Aa$ and 
$\Ef$ is isomorphic to a subsemigroup of $\Aa^\Aa$. 
\item Suppose now that $f\notin \Ef$ so that $\eta=\pi_*(f)(\bar 0)\neq \bar 0$.  
If $\eta = \bar 0 + n$ for some $n\in \Z$ then $f\sigma^{-n}\in\Ef$ and we can apply essentially the same argument as above. So suppose that $\eta$
is not in the orbit of $\bar 0$. 
If $z+\eta\notin O_{\bar 1}$ then $f_z$ is uniquely determined by $\eta$, because $\pi^{-1}(z+\eta)$ contains a single point. If $z+\eta\in O_{\bar 1}$ then $z\notin O_{\bar 1}$ and hence $\pi^{-1}(z)$ contains a single point. Hence $\im f_z$ is a single point. 
By equivariance, $f$ is therefore uniquely determined by $\eta$ together with the unique point in $\im f\cap \pi^{-1}(\bar 1)$; we can view the latter as a choice of symbol from $\Aa$. 
\end{enumerate}
We start by determining $\Ef$, which, as we saw, amounts to determine the possibilities for $f_{\bar 1}$. For that we use an idea from \cite{ESBS} which is based on the equality 
$\theta\circ\sigma = \sigma^5\circ\theta$ together with the fact that $ \left.\tilde\theta\right|_{\pi^{-1}(\bar 1)}= \left.\id^{}\right|_{\pi^{-1}(\bar 1)}$. It implies
$$   \left.\tilde\theta^k\circ \sigma^n\right|_{\pi^{-1}(\bar{1})} = \left.\sigma^{n5^{k}}\right|_{\pi^{-1}(\bar{1})} $$
for all $n\in \Z$ and $k\in\N$.

Any element $f\in\Ef$ is a limit of a generalised sequence $f=\lim \sigma^{n_\nu}$ with $\pi_*(f)(\bar 0) = \bar 0$. By the continuity of $\pi_*$ this implies $\lim n_\nu = \bar 0$. It follows that for all $k$ there is $\nu_k$ such that for all $\nu\succeq \nu_k$ the number $n_\nu$ is divisible by $5^k$. 
In other words, we can factor $n_\nu = m_\nu 5^{k_\nu}$ with $m_\nu\in\Z$, $k_\nu\in\N$ such that $\lim k_\nu=+\infty$. Then
$$ f(x) = \lim \sigma^{m_\nu 5^{k_\nu}}(x) = 
\lim \tilde\theta^{k_\nu} \sigma^{m_\nu}(x)$$
As $x,f(x)\in \pi^{-1}(\bar 1)$, they are uniquely determined by their seed, which is their symbol on $0$. Denoting $ev_0: \pi^{-1}(\bar 1)\to \Aa$ the bijection $ev_0(x)=x_0$ we thus see that $ev_0\circ f\circ ev_0^{-1}\in \Aa^\Aa$ determines uniquely $f$. Hence  $\tilde\theta^{k_\nu} \sigma^{m_\nu}$ can only converge if $ev_0 \circ \sigma^{m_\nu}\circ ev_0^{-1}$ converges. Since $\Aa$ is finite, this means that $ev_0 \circ \sigma^{m_\nu}\circ ev_0^{-1}$ must become constant. We can read off the possible maps $ev_0 \circ \sigma^{m_\nu}\circ ev_0^{-1}$ from  the columns which occur in (\ref{eq-fp}), possibly after further application of $\tilde\theta$. Each column corresponds to a such map. Inspecting (\ref{eq-fp}) we find the following maps 
$$\Ef = \{\id,\Pi_\a,\Pi_\b, \Pi_\c,\phi\}$$
where 
$\Pi_{\sf}$ 
maps all symbols to $\sf$ and $\phi(\a) = \a$, $\phi(\b) = \a$, $\phi(\c) = \b$. 
Since $(x_\a,x_\b)$ is a forward asymptotic pair, no other maps will appear upon iteration of $\tilde\theta$.

All elements but $\phi$ are idempotents. $\Pi_\a,\Pi_\b, \Pi_\c$ are the minimal idempotents. The latter form the minimal two-sided ideal $\M^{fib}=\Ef\cap \M$ of $\Ef$. Their products are
$\Pi_{\sf} \Pi_{\sf'}=\Pi_{\sf}$ thus forming the so-called left zero semigroup of $3$ elements  $LZ_3$.
The products involving $\phi$ are,
$$\phi^2 = \Pi_\a,\quad \phi \Pi_\a = \Pi_\a \phi = \Pi_\a,\quad \phi \Pi_\b = \Pi_\a,\quad \phi \Pi_\c = \Pi_\b, 
\quad  \Pi_{\b} \phi = \Pi_{\b}, \quad  \Pi_{\c} \phi = \Pi_{\c}$$
the first relation showing that $\im\phi^2$ is strictly contained in $\im \phi$ and thus $\phi$ not completely regular.

\subsection{Full Ellis semigroup} The system $(X_\theta,\sigma,\Z)$ is an almost one-to-one extension of its maximal equicontinuous factor. It is a general fact that for those systems the kernel $\M=\M(X,T)$ is a direct product of $\M^{fib}=\M^{fib}(X,T)$ with the maximal equicontinuous factor \cite{Aujogue-Barge-Kellendonk-Lenz}. 
Thus in our present situation 
$$\M \ni f \mapsto (\Pi_{\sf},\pi_*(f)(\bar 0)) \in LZ_3\times \Z_5$$  is an isomorphism of semigroups,
where $\Pi_{\sf}$ is the unique minimal idempotent such that $\Pi_{\sf} f=f$. This isomorphism of semigroups is not a homeomorphism if 
$LZ_3\times \Z_5$ is equipped with the product topology.

The full Ellis semigroup $E$ contains $\M$, a copy of the acting group $\Z$, the element $\phi$ and their possible products. Hence it contains also $\phi\Z$ which not a group. 
Products of elements of $\phi\Z$ land in $\M$.
No element of $\phi\Z$ is completely regular and hence $\phi\Z$ does not intersect $\M$ nor $\Z$. 
 
The above calculation shows that all elements of $\Ef \Z$ are contained in $\M\sqcup \phi\Z\sqcup \Z$ ($\sqcup$ denotes disjoint union). Let $f\in E\backslash \Ef \Z$. Then $\pi_*(f)(\bar 0)\neq \bar 0 + \Z$ so that, 
as we saw above, $\im f\cap \pi^{-1}(\bar 1)$ contains a single point.   
Let $\sf\in\Aa$ such that $x_{\sf}$ is the unique point in  $\im f\cap \pi^{-1}(\bar 1)$. 
Since $\Pi_\sf$ is the identity on all fibres which are not in the orbit of $\bar 1$ we have
 $\Pi_{\sf} f=f$. As $\Pi_{\sf} f\in\M$ we have $f\in \M$. We thus have proven that  
 $$E(X_\theta,\Z) = \M\sqcup \phi \Z\sqcup \Z\cong LZ_3\times \Z_5\sqcup \phi \Z\sqcup \Z.$$
 Since the system is backward almost distal  the backward part of the Ellis semigroup is  
 ${E}(X_\theta,\Z^-) = \M^-\sqcup \Z^-$. 
 This is compatible with the observation that $\phi$ cannot be obtained as a limit $\lim \sigma^{n_\nu}$ with $n_\nu\to -\infty$. 
 On the other hand, $\phi\alpha^n$ belongs to $E^+$, for any $n\in \Z$. Thus  
$E(X_\theta,\Z^+) = \M^+\sqcup \phi \Z\sqcup \Z^+$. Finally, the proximality relation is transitive for $(X_\theta,\Z)$ so that by Thm.~\ref{thm-E1} we have $\M^-=\M^+\cong LZ_3\times \Z_5$.  

\bibliographystyle{abbrv}


\def\ocirc#1{\ifmmode\setbox0=\hbox{$#1$}\dimen0=\ht0 \advance\dimen0
  by1pt\rlap{\hbox to\wd0{\hss\raise\dimen0
  \hbox{\hskip.2em$\scriptscriptstyle\circ$}\hss}}#1\else {\accent"17 #1}\fi}

\end{document}